\documentclass{article}
\usepackage[utf8]{inputenc}
\usepackage{url}
\usepackage{hyperref}
\usepackage{mathtools}  
\usepackage{xfrac} 
\usepackage{amsmath}
\usepackage{parskip}
\usepackage[english]{babel}
\usepackage{amsthm}
\usepackage[english]{babel}
\usepackage{amsthm}
\usepackage{txfonts}

\usepackage{flushend}
\newtheorem{theorem}{Theorem}[section]
\newtheorem{lem}[theorem]{Lemma}
\theoremstyle{definition}
\newtheorem{definition}{Definition}[section]
\newtheorem{Coro}{Corollary}[section]
\theoremstyle{remark}

\newtheorem{Remark}{Remark}[theorem]
\title{On the Bj\"orling problem for Born-Infeld solitons}
\author{ Sreedev Manikoth\thanks{BS-MS Programme, Indian Institute of Science, Education and Research, Pune, Maharashtra, India}}

\usepackage{}
\begin{document}

\maketitle
\begin{abstract}
  The Bj\"orling problem and its solution is a well known result for minimal surfaces in Euclidean three-space. The minimal surface equation is similar to the Born-Infeld equation, which is naturally studied in physics. In this semi-expository article, we ask the question of the  Bj\"orling problem for  Born-Infeld solitons. This begins with the case of locally Born-Infeld soliton surfaces and later moves on to graph-like surfaces. We also present some results about their representation formulae.
\end{abstract}
\section{Introduction}

In general zero-mean curvature surfaces are called minimal surfaces. The theory of minimal surfaces in Euclidean 3-space is well studied with several results. Any non-parametric minimal surface $(x, y, \psi(x,y))$ satisfies the minimal surface equation,

$$(1+\psi_y^2)\psi_{xx}-2 \psi_x \psi_y \psi_{xy}+(1+\psi_x^2)\psi_{yy} =0$$.

This equation is similar to the Born-Infeld equation,

$$(1-\psi_y^2)\psi_{xx}+2 \psi_x \psi_y \psi_{xy}-(1+\psi_x^2)\psi_{yy} =0$$

This motivates us to ask similar questions about the  Born-Infeld solitons. In particular one can ask, can we find an analog of Weierstrass-Ennepper representation formulae for the Born-Infeld solitons?  This was answered by Barbishov and Chernikov, which we would shortly see in this article. We can also ask for an analog of the Bj\"orling problem. In this article we give an answer to that question. We start with the definition of locally Born-Infeld soliton surfaces.

\begin{definition}[Born-Infeld soliton general surfaces]
A  surface is said to be a Born-Infeld soliton general surface if it is locally of the form $(\psi(y,z),y,z)$ , $(x,\psi(x,z),z)$ or $(x,y,\psi(x,y))$ where $\psi$ solves the Born-Infeld equation.
\end{definition}

in \cite{Rr}, R. Dey and R.K. Singh showed that timelike minimal graphs over $y-z$ plane have Born-Infeld solitons as height function. We show that the same result holds for timelike minimal graphs over the $x-z$ plane. We also prove that any timelike minimal surface without singularities is locally a graph over $x-z$ or $y-z$ plane. Thus we conclude timelike minimal surfaces without singularities are Born-Infeld general surfaces.

We use the above result to solve Bj\" orling problem for Born-Infeld general surfaces. Bj\" orling problem for timelike minimal surfaces has already been solved in \cite{Rb} and \cite{Yw}. For regular space or timelike curves, one gets a timelike minimal surface without singularities as the solution to the  Bj\" orling problem. This implies that for regular space or timelike curves Bj\" orling problem for Born-Infeld general surfaces can be solved.

In \cite{Ea} E.A Paxton showed that any compact subset of a global properly immersed timelike minimal surface is a timelike minimal graph over some timelike plane. We generalize the result in \cite{Rr} to show timelike minimal graphs over any timelike plane have Born-Infeld soliton as height function. Moreover, we see that for regular real analytic curve $c$ and vector field $n$ which have entire split-holomorphic functions as an analytic extension we get a global properly immersed timelike surface as the solution to the Bj\"orling problem. So solving timelike Bj\"orling problem for such real analytic strips and looking at their compact subsets is a way to find Born-Infeld solitons. In particular, we note that for real analytic strips of regular real-analytic  curves and unit vector fields $(c,n)$ with components as polynomials in $t$, one can use the above result to find Born-Infeld solitons.

 Graph-like Born-Infeld solitons are of special interest to physicists and we try to ask a similar question for them. From \cite{Rr} we note that spacelike minimal graphs and timelike minimal graphs over $y-z$ plane are Born-Infeld soliton graphs. We use this idea and ask for what kind of real analytic strips $(c,n)$, do we get a spacelike minimal graph or timelike minimal graph as a solution to the Bj\"orling problem. We characterize such curves and normals for which graphical solutions can or cannot be found.

Lastly, we go through the representation formulae given by Barbishov and Chernikov. We show that the Barbishov and Chernikov representation formula, like the Weierstrass-Enneper representation, fails  at zero Gauss curvature points.  In \cite{Lm} L. McNertney showed that any surface in $\mathbb{L}^3$ which can be expressed as the sum of two lightlike curves with linearly independent velocities is timelike minimal.  We see that  the Barshiov and Chernikov representation formula also expresses the surface as a sum of two lightlike curves with linearly independent velocities, implying these Born-Infeld solitons are timelike minimal in $\mathbb{L}^3$. Also, normal vector fields of these surfaces parametrized by $r-s$ coordinates are the same.

We note that A. Das also in \cite{Ar} independently solved the Bj\"orling problem for  Born-Infeld solitons, $X(\omega(t))=c(t)$,$N(\omega(t))=n(t)$ where $\omega(t)$ is a curve in $r-s$ plane determined by $c$ and $n$. They also shows that Bj\"orling problem for Born-Infeld solitons may not have unique solutions.

Our results here are dependent on plenty of earlier work done by several mathematicians. We would have a look at those as we go through them. 
.

As a summary, we aim to describe three results. We would first prove that any timelike minimal surface without singularities is locally  the graph of a Born-Infeld soliton over $y-z$ or $x-z$ plane. This answers the Bj\"orling problem for surfaces that are locally Born-Infeld solitons, in the special case when the curve is assumed to be regular. The third section is about some corollaries of E.A Paxton's results (\cite{Ea}). In the fourth section, we  deal with the Bj\"orling problem of surfaces that are globally Born-Infeld solitons and present some results. In the last few sections, we study a special class of Born-Infeld solitons, given by Barbishov and Chernikov and we would prove some theorems about them.

Throughout this article we will be using the following definition for the Lorentz-Minkowski space, $\mathbb{L}^3$.

\begin{definition}
$\mathbb{L}^3$ is $\mathbb{R}^3$ with the metric $ds^2=dx^2+dy^2-dz^2$
\end{definition}
\section{Regular timelike minimal surfaces and Born-Infeld solitons}

We will start with defining Born-Infeld solitons.

\begin{definition}[Born-Infeld soliton]
Let $\Omega \subset $ $\mathbb{R}^2$ be an open subset. Let $(u,v)$ $\in$ $\Omega$. Now we will denote this subset by $\Omega_{(u,v)}$.A map $\phi$: $\Omega_{(u,v)}$ $\rightarrow$ $\mathbb{R}$ is said to be a Born-Infeld soliton if it solves the Born-Infeld equation in the variables $u,v$. 
\end{definition}

First we will show a lemma about timelike minimal graphs over $x-y$ plane.

\begin{lem}
Any timelike minimal graph $X(x,y)=(x,y,\phi(x,y))$ without singularities is locally is a graph of the form $(x,\psi(x,z),z)$ or $(\psi(y,z),y,z)$ \medskip for some Born-Infeld soliton $\psi$.
\end{lem}
\begin{proof}Here the Jacobian of the surface $X$ at a point p looks like this,

$$\begin{pmatrix}
1 & 0 \\
0 & 1 \\
 \phi_x(p) & \phi_y(p)
\end{pmatrix}$$

Thus note that $$\begin{pmatrix}
1 & 0 \\
 \phi_x(p) & \phi_y(p)
\end{pmatrix}$$  or $$\begin{pmatrix}
0 & 1 \\
 \phi_x(p) & \phi_y(p)
\end{pmatrix}$$  has rank 2 only if their determinants $\phi_x(p)$ or $\phi_y(p)$ is nonzero . In other words, our surface is always locally a graph over $x-z$ or $y-z$ plane if $\phi_x(p)$ or $\phi_y(p)$ is nonzero for any point p. This is always true for a timelike minimal graphs without singularities over $x-y$ plane as they satisfies $$\phi_x ^2(p)+\phi_y^2(p)>1$$ at all points as the  normal of X, $$N=\frac{(-\phi_x,-\phi_y,-1)}{\sqrt{|\phi_x^2+\phi_y^2-1|}} $$ is spacelike.  To show height function of local $x-z$ or $y-z$ graph is a Born-Infeld soliton, one can compute mean curvature and equate it to zero. This part is similar to  R. Dey and R.K. Singh's  proof of  height functions of timelike minimal graphs without singularities over $y-z$ planes are Born-Infeld solitons(in \cite{Rr} pages 528 to 530)
\end{proof}

Now we will use this lemma to  prove that timelike minimal surfaces without singularities are Born-Infeld soliton general surfaces.
\begin{theorem}
Any timelike minimal surface without singularities is locally a graph of a function over the x-z or y-z plane with their height function being a Born-Infeld soliton.
\end{theorem}
\begin{proof}
Any regular timelike minimal surface is  locally a graph and it is of the form $(x,y,\psi(x,y))$, or $(x,\psi(x,z),z)$ or  $(\psi(y,z),y,z)$. By lemma 2.1, we note that the timelike $x-y$ graph without singularities is also locally a $y-z$ or $x-z$ graph. Using zero-mean curvature condition  one can conclude that height function must be a Born-Infeld soliton.
\end{proof}
\begin{Remark}
Alternatively in \cite{Rl} proposition 3.3(page 75), R. Lopez proved that any timelike minimal surface is locally a graph over $x-z$ or $y-z$ plane by noting that components of the normal are Jacobian of the map from $y-z$,$x-z$ and $x-y$ plane into the image of the surface.   
\end{Remark}

In theorem 3.3 of \cite{Yw}(Page 1091),  Y.W Kim, S.E Koh and S-E Yang proved that if $c$ is a regular spacelike or timelike curve, then there is a timelike minimal surface without singularities solving the Bj\"orling problem. This gives us the following result.

\begin{Coro}
If $c$ is a regular real analytic spacelike or timelike curve and $n$ a real analytic spacelike normal vector field, then there exists a Born-Infeld soliton general surface which solves the Bj\"orling problem.
\end{Coro}

Another corollary of theorem 2.1 is the following.

\begin{Coro}
Let $c:I \rightarrow \mathbb{L}^3$ be a regular timelike curve in $\mathbb{L}^3$ such that $c^{''}(t)$ is spacelike for all $t \in I$. Then there exist a Born-Infeld soliton general surface containing c as geodesic(by a geodesic here, we mean a curve such that principal normal agreeing with surface normal in $\mathbb{L}^3$).

\end{Coro}

\begin{proof}
Here since our curve is regular we can give it arc length parametrization which has a constant speed. Now we refer to corollary 3.2, page 489 of \cite{Rb}. We also note that for a regular curve we get a timelike minimal surface without singularities as solutions to the Bj\"orling problem which is also a Born-Infeld soliton general surface.
\end{proof}

\section{On compact subsets of timelike minimal surfaces}
In this section, we get some corollaries of theorem 1.1 in \cite{Ea}. We would first prove a lemma.

\begin{lem}

Let $X:\Omega \rightarrow \mathbb{R}^3$ be a timelike minimal surface which is a smooth graph over a timelike plane $P$. Choose a orthonormal basis $\{b_2,b_3\}$ with respect to  $\langle , \rangle_{\mathbb{L}^3}$ for the plane $P$, with $b_2$ a spacelike vector ,$b_3$ a timelike vector . Also, let $b_1=N$ be the unit spacelike surface normal of the timelike plane $P$.Then $\{b_1,b_2,b_3\}$ forms a orthonormal basis of $\mathbb{L}^3$. For any $x_2 b_2+x_3 b_3$ in P, we can consider $$\psi(x_2,x_3)=\langle X,N \rangle_{\mathbb{L}^3}.$$ Then such a $\psi$ satisfies the Born-Infeld equation in variables $x_2,x_3$ and thus is a Born-infield soliton.
\end{lem}

\begin{proof}
Since $X$ is a smooth graph over plane $P$ we know that the projection map,  $$\pi: X(\Omega) \rightarrow R^2$$ with $\pi(p)$ being the  projection of the point $X(p)$  onto the plane $P$=span$\{b_2,b_3\}$  is a diffeomorphism. Let $\Sigma \subset P$ be the image of this map. Using $$\phi=\pi^{-1} : \Sigma \rightarrow X(\Omega) $$ we get a graph $X(x_2,x_3)=\psi(x_2,x_3)N+x_2 b_2+x_3 b_3$ with $\psi(x_2,x_3)=\langle X(x_2,x_3),N \rangle_L$. Here  $X(x_2,x_3) :P \rightarrow \mathbb{R}^3$  is  $X \circ \phi $. To show that such a map $\psi(x_2,x_3)$ satisfies the Born-infield equation, one can compute the mean curvature in this new parametrization and equate it to zero.
\end{proof}
Now we state the main result of this section.
\begin{theorem}
If  $(c,n)$ is a smooth real analytic strip with the properties that, 
\begin{itemize}
  \item $c$, a  regular curve 
  \item $c$ and $n$ have entire split-holomorphic functions as analytic extension
\end{itemize}

then any compact subset of the timelike minimal surface solving this Bj\"orling problem is a timelike minimal graph over some timelike plane with height function a Born-Infeld soliton.
\end{theorem}

\begin{proof}
 We first note that since $c$ and $n$ have analytic extensions which are entire functions, $\Omega$ can be taken to be $\mathbb{R}^2$. Since $c$ is regular,  solution to Bj\" orling problem is a regular surface(we refer to theorem 3.3 of \cite{Yw}). Thus we have a smooth $X$: $\mathbb{R}^2$ $\rightarrow$ $\mathbb{R}^3$ which is a properly immersed surface, solving this Bj\"orling problem. Now by theorem 1.1 (page 3036) of \cite{Ea} any compact subset of this surface is a timelike graph over some timelike plane. Using lemma 3.1, we conclude that height function is a Born-Infeld soliton.
\end{proof}

In fact for any $M>0$, when restricted to a diamond $D_M$=$\{(u,v)||u|+|v| \leq  M\}$  We would still get a timelike graph over some timelike plane. This was used in the proof of theorem 1.1 (page 3036) of \cite{Ea}.

Now we state a special case of theorem 3.2.

\begin{Coro}
If  $(c,n)$ is a smooth real analytic strip with the properties that, 
\begin{itemize}
  \item $c$, a  regular curve 
  \item $c(t)$ and $n(t)$ have components as  polynomials in $t$  with real coefficients.

\end{itemize}

for any $M>0$, solution to the time like Bj\" orling problem when restricted to a diamond $D_M$=$\{(u,v)||u|+|v| \leq  M\}$ is a smooth graph over some timelike plane with height function being a Born-Infeld soliton.
\end{Coro}

\section{Bj\" orling problem for graph-like Born-Infeld solitons}

We first define graphical Born-Infeld solitons.
\begin{definition}[Born-Infeld soliton surface]
A surface $X$ is said to be a Born-Infeld soliton surface if it is of the form $X(y,z)=(\psi(y,z),y,z)$ for some Born-Infeld soliton .
\end{definition}

Let us recall the definition of positive quasidefinite matrix from \cite{Dh}.

\begin{definition}
A matrix $J$ is said to be positive quasidefinite if $$A=\frac{J+J^{T}}{2}$$ is positive definite.
\end{definition}
 We asked for what kind of real-analytic strips $(c,n)$ the Bj\"orling problem gives a graphical solution. We characterized set of real-analytic strips $(c,n)$ for which one gets a time-like minimal graph as solution to the Bjorling problem. Note that in the following result, we are using split-complex analysis instead of complex-analysis. Here $z=x+k^{'}y$ with $k^{{'}^{2}}=1$. We refer to \cite{Rb} for more about split-complex analysis. 

\begin{theorem}
Given a real analytic curve $c(t)=(c_1(t),c_2(t),c_3(t)) $ which is timelike or spacelike in $L^3$  and a real analytic spacelike  normal $n=(n_1(t),n_2(t),n_3(t))$  let, 

$$J_{(c,n)}(t)=\begin{vmatrix}
\frac{c_{2_{u}}(t)}{2} & \frac{c_{2_{v}}(t)}{2}+(n(t) \times c^{'}(t))_2\\
 \frac{c_{3_{u}}(t)}{2} & \frac{c_{3_{v}}(t)}{2}+ (n(t) \times c^{'}(t))_3
\end{vmatrix}$$

be a real-valued function defined on  $I$.

Fix a real analytic strip $(c,n)$ and let $\Omega_{(c,n)}$ be a  domain  where analytic extension of both $c$ and $n$ exists.

1) $t \rightarrow J_{(c,n)}(t)$ has a zero in $I$ $\implies$ there does not exist a solution for Bj\"orling problem of time like Born-Infeld soliton surfaces without singularities.

2) If $\Omega_{(c,n)}$ is convex, and if $$J_{(c,n)}(z)=\begin{pmatrix}
\frac{c_{2_u}}{2}+(Im(n(z) \times c^{'}(z))_2   &\frac{c_{2_v}}{2}+(Re(n(z) \times c^{'}(z))_2 \\
 \frac{c_{3_u}}{2}+(Im(n(z) \times c^{'}(z))_3 & \frac{c_{3_v}}{2}+(Re(n(z) \times c^{'}(z))_3)
\end{pmatrix}$$
 has a  non-vanishing determinant and is positive quasidefinite  for all $z$ in $\Omega_{(c,n)}$, then there is a timelike Born-Infeld soliton surface without singularities in $\mathbb{L}^3$ as solution  to the bj\"orling problem and it is given by 
$$X(z)=Re\{c(z)+ k'\int_{t_0}^{z} n(w) \times c^{'}(w) \,dw \}$$

if $c$ is timelike and,

$$X(w)=Re\{c(w)+ k'\int_{t_0}^{w} n(\zeta) \times c^{'}(\zeta) \,d\zeta \}$$

When $c$ is spacelike(Here $w=k^{'}z$).

\end{theorem}

\begin{proof}

We present the proof for the case when $c$ is timelike. When $c$ is spacelike, the proof is similar.

The main idea of the proof is to  use the implicit function theorem to understand when does the timelike minimal surface solution to the Bj\" orling problem becomes a graph over the y-z plane. 

For Bj\"orling problem for timelike minimal surfaces,we know the solution is $$X(z)=Re\{c(z)+ k'\int_{t_0}^{z} n(w) \times c^{'}(w) \,dw \}$$(We refer to \cite{Rb}, page 485, theorem 3.1).

Let $$F(z)=c(z)+ k'\int_{t_0}^{z} n(w) \times c^{'}(w) \,dw $$.

Then $$X(z)=\frac{F(z)+\overline{F(z)}}{2}$$

$$\frac{\partial X}{\partial z}=\frac{1}{2}\frac{\partial F}{\partial z}=\frac{1}{2}(\frac{\partial c}{\partial z}+k'(n(z) \times c^{'}(z))).$$

We note that for split complex numbers,
$$\frac{\partial X}{\partial z}=\frac{1}{2}\left({\frac{\partial X}{\partial u}+k'\frac{\partial X}{\partial v}}\right).$$

Thus $$\frac{\partial X}{\partial u}=Re(2\frac{\partial X}{\partial z})=\frac{1}{2}\frac{\partial c}{\partial u}+Im(n(z) \times c^{'}(z))$$

$$\frac{\partial X}{\partial v}=Im(2\frac{\partial X}{\partial z})=\frac{1}{2}\frac{\partial c}{\partial v}+Re(n(z) \times c^{'}(z))$$

$$\begin{pmatrix}
X_u & X_v 
\end{pmatrix}=\begin{pmatrix}
\frac{c_u}{2}+Im(n(z) \times c^{'}(z) & \frac{c_v}{2}+ Re(n(z) \times c^{'}(z)
\end{pmatrix}.$$

Which implies,
$$\begin{pmatrix}
x_u & x_v\\
y_u & y_v \\
 z_u & z_v
\end{pmatrix} =\begin{pmatrix}
\frac{c_{1_u}}{2}+(Im(n(z) \times c^{'}(z))_1   &\frac{c_{1_v}}{2}+(Re(n(z) \times c^{'}(z))_1\\
\frac{c_{2_u}}{2}+(Im(n(z) \times c^{'}(z))_2   &\frac{c_{2_v}}{2}+(Re(n(z) \times c^{'}(z))_2 \\
 \frac{c_{3_u}}{2}+(Im(n(z) \times c^{'}(z))_3 & \frac{c_{3_v}}{2}+(Re(n(z) \times c^{'}(z))_3)
\end{pmatrix}.$$ 

Thus
$$\begin{pmatrix}
y_u & y_v \\
 z_u & z_v
\end{pmatrix} =\begin{pmatrix}
\frac{c_{2_u}}{2}+(Im(n(z) \times c^{'}(z))_2   &\frac{c_{2_v}}{2}+(Re(n(z) \times c^{'}(z))_2 \\
 \frac{c_{3_u}}{2}+(Im(n(z) \times c^{'}(z))_3 & \frac{c_{3_v}}{2}+(Re(n(z) \times c^{'}(z))_3)
\end{pmatrix}.$$ 

Let $$J_{(c,n)}(z)=\begin{vmatrix}
\frac{c_{2_u}}{2}+(Im(n(z) \times c^{'}(z))_2   &\frac{c_{2_v}}{2}+(Re(n(z) \times c^{'}(z))_2 \\
 \frac{c_{3_u}}{2}+(Im(n(z) \times c^{'}(z))_3 & \frac{c_{3_v}}{2}+(Re(n(z) \times c^{'}(z))_3)
\end{vmatrix}.$$ 

Thus whenever $J_{(c,n)}(z)$ is non vanishing, by implicit function theorem one can represent the surface as locally a graph over $y-z$ plane. The real valued function $t \rightarrow J_{(c,n)}(t)$ mentioned in theorem 

$$J_{(c,n)}(t)=\begin{vmatrix}
\frac{c_{2_{u}}(t)}{2} & \frac{c_{2_{v}}(t)}{2}+(n(t) \times c^{'}(t))_2\\
 \frac{c_{3_{u}}(t)}{2} & \frac{c_{3_{v}}(t)}{2}+ (n(t) \times c^{'}(t))_3
\end{vmatrix}$$

is  the restriction of this map $z \rightarrow J_{(c,n)}(z)$ to $I$. Using continuity arguments one can conclude that $J_{(c,n)}(z)$ non-vanishing in a some allowed domain $\Omega_{(c,n)}$ is equivalent to $J_{(c,n)}(t)$ non-vanishing on $I$. This completes the proof of the first part of the theorem.

Now if  $\psi: (u,v) \rightarrow (y(u,v),z(u,v)) $ is injective as well, it would become a diffeomorphism. In \cite{Dh} theorem 6 (Page 88) D. Gale and H.Nikaido shows that if $\Omega$ is a convex domain and if $\psi:\Omega \rightarrow \mathbb{R}^2$ has a positive quasidefinite Jacobian at all points, then $\psi$ is injective. The second condition ensures this.
\end{proof}

Now we would just state when can we get a spacelike minimal graph or a spacelike Born-Infeld soliton surface without singularities as the solution to the Bjorling problem. The proof is similar, except one has to use complex-numbers and complex analysis.

\begin{theorem}
Given a real analytic curve $c(t)=(c_1(t),c_2(t),c_3(t)) $  in $L^3$  and a real analytic timelike  normal $n=(n_1(t),n_2(t),n_3(t))$  let, 

$$J_{(c,n)}(t)=\begin{vmatrix}
\frac{c_{2_{u}}(t)}{2} & \frac{c_{2_{v}}(t)}{2}-(n(t) \times c^{'}(t))_2\\
 \frac{c_{3_{u}}(t)}{2} & \frac{c_{3_{v}}(t)}{2}- (n(t) \times c^{'}(t))_3
\end{vmatrix}$$

be a real-valued function defined on  $I$.

Fix a real analytic strip $(c,n)$ and let $\Omega_{(c,n)}$ be a  domain  where analytic extension of both $c$ and $n$ exists.

1) $t \rightarrow J_{(c,n)}(t)$ has a zero in $I$ $\implies$ there does not exist a solution for Bj\"orling problem of space like Born-Infeld soliton surfaces without singularities.

2) If $\Omega_{(c,n)}$ is convex, and if $$J_{(c,n)}(z)=\begin{pmatrix}
\frac{c_{2_u}}{2}+(Im(n(z) \times c^{'}(z))_2   &\frac{c_{2_v}}{2}-(Re(n(z) \times c^{'}(z))_2 \\
 \frac{c_{3_u}}{2}+(Im(n(z) \times c^{'}(z))_3 & \frac{c_{3_v}}{2}-(Re(n(z) \times c^{'}(z))_3)
\end{pmatrix}$$
 has a  non-vanishing determinant and is positive quasidefinite  for all $z$ in $\Omega_{(c,n)}$  then there is a  spacelike Born-Infeld general surface, without singularities in $\mathbb{L}^3$ as solution  to the Bj\"orling problem and it is given by 
$$X(z)=Re\{c(z)+ i\int_{t_0}^{z} n(w) \times c^{'}(w) \,dw \}$$

\end{theorem}

\begin{Remark}
 A. Das in \cite{Ar} show that the above positive-quasidefinite condition can be replaced with J(c,n)(t) being a P-matrix. 
\end{Remark}

\section{On solutions given by Barbishov and Chernikov }
 We refer to pages 617 to 619 of \cite{Gb} for the representation formula of Born-Infeld soliton surfaces given by Barbishov and Chernikov. We show that it holds at any non-zero Gauss curvature point.

\begin{theorem}
For any timelike Born infeld soliton surface $(\psi(y,z),y,z)$ without singularities, for any non-zero Gauss curvature point there is an open neighbourhood with  two $C^2$ functions $F$ and $G$ such that the surface can be represented there as,
$$y-z=F(r)-\int s^2 G^{'}(s) ds$$
$$y+z=G(s)-\int r^2 F^{'}(r) dr$$
$$\psi(y,z)=\int r F^{'}(r) dr+\int s G^{'}(s) ds.$$
Conversely, any graph-like surface $(\psi(y,z),y,z)$ represented this way is a Born-Infeld soliton surface.

\end{theorem}

\begin{proof}
We refer to \cite{Gb}, pages 617-619(Section 17.15) for the proof of the above representation formulae. Note the proof starts with the assumption $$\psi_{y}^2-\psi_{z}^2+1>0.$$ This implies the surface is timelike. Also Note that in the proof to go from step 17.89 to 17.90, there was an interchange of the roles of dependent and independent variables. For this, we want the map $$\psi :(\xi,\eta) \rightarrow (u,v)$$ to be a diffeomorphism. So  we need the Jacobian of this map $\psi$ to be nonzero. This implies, at such points $p$  $$(\psi_{yy} \psi_{zz}- \psi_{yz}^2)(p) \neq 0$$.

The above condition is equivalent to $\psi$ being a local-diffeomorphism. Note that Gauss curvature of $(\psi(y,z),y,z)$ is given by,
$$K(p)=\frac{\psi_{yy}\psi_{zz}-\psi_{yz}^2}{(\psi_y^2-\psi_z^2+1)^2}.$$
Thus if there are no singularities, $$(\psi_{yy} \psi_{zz}- \psi_{yz}^2)(p) \neq 0$$ is equivalent to saying that Gauss curvature $K(p)$ is nonzero.
\end{proof}

Now we state a result about their surface normal.

\begin{theorem}
The surface normal $N(r,s)$ (in $\mathbb{L}^3$)  of all  Born-Infeld soliton surfaces described by the representation formula of Barbishov and Chernikov are the same.
\end{theorem}

\begin{proof}
The proof follows by computation.

$$X_r=\left(rF^{'}(r), \hspace{0.25cm}\frac{F^{'}(r)(1-r^2)}{2}, \hspace{0.25cm}\frac{-F^{'}(r)(1+r^2)}{2}\right).$$
$$X_s=\left(sG^{'}(s),\hspace{0.25cm}\frac{G^{'}(s)(1-s^2)}{2},\hspace{0.25cm}\frac{G^{'}(s)(1+s^2)}{2}\right).$$

$$N(r,s)=\frac{X_r \times X_s}{|X_r \times X_s|}=\left(\frac{r+s}{1+rs},\hspace{0.25cm} \frac{r-s}{1+rs},\hspace{0.25cm} \frac{rs-1}{1+rs}\right).$$

Thus $N(r,s)$ is independent of F and G.
\end{proof}
Now we give a geometric interpretation of the above formula.

\begin{theorem}
Graphical surfaces $(\psi(y,z),y,z)$ with $\psi(y,z)$,$y$,$z$ as described by the representation formula of Barbishov and Chernikov can be written as,

$$X(r,s)=\frac{\psi(r)+\phi(s)}{2}$$

with $$\psi(r)=\left(2\int r F^{'}(r) dr, \hspace{0.25cm}F(r)-\int r^2 F^{'}(r) dr, \hspace{0.25cm}-F(r)-\int r^2 F^{'}(r) dr\right),$$

$$\phi(s)=\left(2\int s G^{'}(s) ds,\hspace{0.25cm}G(s)-\int s^2 G^{'}(s) ds, \hspace{0.25cm}G(s)+\int s^2 G^{'}(s) ds\right),$$

such that $\psi,\phi$ lightlike curves in $\mathbb{L}^3$ with $\psi^{'}(r)$ and $\phi^{'}(s)$ are linearly independent for all values of r and s. This implies these surfaces are timelike minimal.
\end{theorem}

\begin{proof}
It follows from computation that $$X(r,s)=\frac{\psi(r)+\phi(s)}{2}$$ and $\psi,\phi$ are lightlike curves. To show $\psi^{'}(r)$ and $\phi^{'}(s)$ are linearly independent for all values of r and s, note that $\psi^{'}(r)=2X_r$ and $\phi^{'}(s)=2X_s$ .

$$X_r=\left(rF^{'}(r), \hspace{0.25cm}\frac{F^{'}(r)(1-r^2)}{2}, \hspace{0.25cm}\frac{-F^{'}(r)(1+r^2)}{2}\right).$$
$$X_s=\left(sG^{'}(s),\hspace{0.25cm}\frac{G^{'}(s)(1-s^2)}{2},\hspace{0.25cm}\frac{G^{'}(s)(1+s^2)}{2}\right).$$

Since $X(r,s)$ is a given to be the graph of a function, which is a regular surface, $X_r$ and $X_s$ are linearly independent for all values of $r$ and $s$. 

Fact 2.2 in \cite{Sa}(Page 541) confirms these surfaces are timelike minimal.
\end{proof}
\section*{Acknowledgement}
I would like to thank my MS thesis supervisor Prof. Rukmini Dey(ICTS Bangalore) and my friend Arka Das(IISC Bangalore) for all the interesting conversations we had during the project. I also thank Rukmini Dey for suggesting the problem to me. I want to express my gratitude to ICTS Bangalore for their hospitality. This work was done as a part of my master's thesis project in IISER Pune.

\end{document}